\documentclass{amsart}
\usepackage[numbers]{natbib}
\newtheorem{theorem}{Theorem}[section]

\newtheorem{lemma}[theorem]{Lemma}

\usepackage{color}
\usepackage[colorlinks, linkcolor=reference,
            citecolor=citation,
          urlcolor=e-mail]{hyperref}

\definecolor{e-mail}{rgb}{0,.40,.80}
\definecolor{reference}{rgb}{.20,.60,.22}
\definecolor{citation}{rgb}{0,.40,.80}

\theoremstyle{definition}
\newtheorem{definition}[theorem]{Definition}

\theoremstyle{remark}

\numberwithin{equation}{section}
\def \l {\langle}
\def \r {\rangle}

\title[Corrigendum: Efffective differential Nullstellensatz]{Corrigendum to the paper:\\
``On bounds for the effective differential Nullstellensatz'' \\
J. of Algebra {\bf 449} (2016) 1--21.}
\author{Omar Le\'on S\'anchez} 
\email{oleonsan@math.mcmaster.ca} 
\address{School of Mathematics, University of Manchester,
Oxford Road,
Manchester, UK, M13 9PL}
\author{Alexey Ovchinnikov}
\thanks{A. Ovchinnikov was partially supported by the NSF grants CCF-0952591, CCF-1563942, and DMS-1413859.}
\email{aovchinnikov@qc.cuny.edu}
\address{Department of Mathematics,
CUNY Queens College,
65-30 Kissena Blvd,
Queens, NY 11367, USA\\
Ph.D. Program in Mathematics, CUNY Graduate Center,
365 Fifth Avenue, New York, New York  10016, USA}

\date{\today}

\begin{document}

\maketitle

\begin{abstract}
We correct a small gap found in the paper. This gap is due to an inequality that does not generally hold. However, under one additional assumption, it does hold. In this note, we provide a detailed proof of this. We then point out that this assumption is satisfied in all instances in which the inequality was used.
\end{abstract}

\

In the proof of Lemma 3.12 we used inequality (3.5). This inequality says that for any $a_1,a_2,b_1,b_2$, nonnegative integers with $a_i\leq b_1$ for $i=1,2$, then, given a positive integer $d$, we have
$$a_1+a_2\leq b_1+ b_2 \implies a_1^{\l d\r}+a_2^{\l d\r}\leq b_1^{\l d\r}+b_2^{\l d\r}$$
However, this does not generally hold\footnote{We are grateful to William Sit for bringing this to our attention.}. For instance, let $a_1=4$, $a_2=4$, $b_1=6$, $b_2=2$, and $d=3$, then $4^{\l 3\r}=5$, $6^{\l 3\r}=7$ and $2^{\l 3\r}=2$; and so $a_1^{\l d\r}+a_2^{\l d\r}=10$ while $b_1^{\l d\r}+b_2^{\l d\r}=9$. This (counter-)example also applies to Lemma 3.12, so this lemma is incorrect as stated.

The purpose of this note is to fix these incorrect statements by including an additional assumption; namely, $b_1$ above should be of the form $\binom{m-1+d}{d}$ for some $m$. We note that, at those places at which (3.5) and Lemma 3.12 were used, this assumption is satisfied, and so all the other statements of the paper remain valid.

Let us briefly recall some of the terminology. Given two nonnegative integers $a$ and $b$, with $a\geq b$, we set
$$\binom{a}{b}=\frac{a!}{(a-b)!\cdot b!}$$
In case $a<b$, we set $\binom{a}{b}=0$. 

We recall that, for $c$ and $d$ positive integers, we have the identity
\begin{equation}\label{bin}
\binom{c+d}{d}=\binom{c-1+d}{d}+\binom{c-1+d}{d-1}.
\end{equation}
By \cite[Lemma~4.2.6]{BH98}, for all positive integers $a$ and $d$, there exist unique $k_d>k_{d-1}>\ldots >k_j\geq j\geq 1$ such that  
\begin{equation}\label{rep2}
a=\binom{k_d}{d}+\binom{k_{d-1}}{d-1}+\ldots+\binom{k_j}{j},
\end{equation}
where $k_d$ is the largest integer such that $a\geq \binom{k_d}{d}$.
\begin{definition}
We call the $k_d,\dots,k_j$ the (Macaulay) $d$-binomial representation of $a$. 
\end{definition}

Macaulay's function $*^{\l d \r}:\mathbb N\to \mathbb N$ is defined as
$$a^{\l d\r}=\binom{k_d+1}{d+1}+\binom{k_{d-1}+1}{d-1+1}+\ldots+\binom{k_j+1}{j+1},$$
where $k_d,\dots,k_j$ is the $d$-binomial representation of $a$. We set $0^{\l d\r}=0$. 

We have: 
\begin{equation}\label{ineq}
a< b \implies  a^{\l d\r}<b^{\l d\r} 
\end{equation}
Indeed, by \cite[Lemma~4.2.7]{BH98} 
$$ a< b \quad\iff\quad (k_d,\dots,k_j)<_{\operatorname{lex}}(k'_d,\dots,k'_i),$$
where $<_{\operatorname{lex}}$ is the (left) lexicographic order and $(k_d,\dots,k_j)$ (respectively, $(k'_d,\dots,k'_i)$) is the $d$-binomial representation of $a$ (respectively, $b$).

By  \eqref{bin} and \eqref{rep2},  for all $a>0$ and $d>0$, there exists a unique  $c>0$ and there exists $A$ such that 
\begin{equation}\label{form1}
a=\binom{c-1+d}{d}+A \quad \text{ with } 0\leq A < \binom{c-1+d}{d-1}.
\end{equation}
Indeed,  let $c=k_d-d+1$. If $c$ were not unique, we would have $c'\neq c$ and $A'$ such that $a=\binom{c'-1+d}{d}+A'$ with $0\leq A'<\binom{c'-1+d}{d-1}$, but then, writing the $(d-1)$-binomial representations of $A$ and $A'$, we would obtain two distinct $d$-binomial representations of $a$; however, we know such representation is unique and \[ A=\binom{k_{d-1}}{d-1}+\ldots+\binom{k_j}{j},\] where $(k_d,\dots,k_j)$ is the $d$-binomial representation of $a$. Since $k_d=c-1+d$, then  $A< \binom{k_d}{d-1}=\binom{c-1+d}{d-1}$ as desired. Indeed, otherwise, for $B:=A-\binom{k_d}{d-1}$, we get \[a= \binom{k_d}{d}+A= \binom{k_d}{d}+\binom{k_d}{d-1}+B=\binom{k_d+1}{d}+B,\] where the letter equality uses \eqref{bin}, which implies $a\geq \binom{k_d+1}{d}$, and this contradicts the choice of $k_d$.

It follows from the definition that
$$a^{\l d\r}=\binom{c+d}{d+1}+A^{\l d-1\r}.$$
Moreover,  for all $a>0$ and $d>0$, there exists a unique  $c\geq 0$ such that
\begin{equation}\label{form2}
a=\binom{c-1+d}{d}+A \quad \text{ with } 0< A \leq \binom{c-1+d}{d-1}.
\end{equation}
Indeed, first, we may assume that $a=\binom{b-1+d}{d}$ for some $b>0$ (otherwise, it has the form as in \eqref{form1}, and we have already justified uniqueness in this case). Then, by \eqref{bin}, $a=\binom{b-2+d}{d}+\binom{b-2+d}{d-1}$. From this, we see that we can set \[c=b-1\quad \text{and}\quad A=\binom{b-2+d}{d-1}.\] Let us justify that this is the only $c$ that works. Indeed, let $c'\neq c$ be such that there exists $A'$ such that  \[a=\binom{c'-1+d}{d}+A'\quad\text{and}\quad 0<A' \leq \binom{c'-1+d}{d-1}.\] If $c'> c$, then, since $c=b-1$ we get $c'-1+d>c-1+d=b-2+d$, and so $c'-1+d\geq b-1+d$ which implies \[\binom{c'-1+d}{d}\geq \binom{b-1+d}{d}.\] As $A'>0$, we get \[a=\binom{c'-1+d}{d}+A'> \binom{b-1+d}{d}=a,\] which is impossible. Finally, if $c'<c$, then 
\begin{align*}\binom{c'-1+d}{d}+A'&\leq \binom{c'-1+d}{d}+\binom{c'-1+d}{d-1}\\
&<\binom{b-2+d}{d}+\binom{b-2+d}{d-1}=\binom{b-1+d}{d}=a,
\end{align*} which contradicts the choice if $c'$.

Also, by \eqref{bin}, in this case, we get the identity
$$a^{\l d\r}=\binom{b-1+d}{d+1}+\binom{b-1+d}{d}.$$
Therefore, regardless of the form we write $a> 0$, either as in \eqref{form1} or \eqref{form2}, we get
\begin{equation}\label{use}
a^{\l d\r}=\binom{c+d}{d+1}+A^{\l d-1\r}.
\end{equation}

\begin{lemma}\label{rep}
Let $m$ and $d$ be a positive integers. If $a$ and $b$ are nonnegative integers, then
\begin{enumerate}
\item $a^{\l d\r}+b^{\l d\r}\leq (a+b)^{\l d\r} $.
\item If, additionally, $\max (a,b)\leq\binom{m-1+d}{d}$ and $c$ is a positive integer such that
\begin{equation}\label{eq:abc} a+b\leq \binom{m-1+d}{d}+c,
\end{equation} 
then 
\[a^{\l d\r}+b^{\l d\r}\leq \binom{m-1+d}{d}^{\l d\r}+c^{\l d\r}.\]
\end{enumerate}
\end{lemma}
\begin{proof}
We assume $a\geq b$. We may assume that $b>0$; indeed, if $b=0$ then, since $a\leq \binom{m-1+d}{d}$, \eqref{ineq} yields $a^{\l d\r}\leq \binom{m-1+d}{d}^{\l d\r}$ which shows (2). In (2), we may also assume that \[a<\binom{m-1+d}{d}\quad \text{and} \quad a+b= \binom{m-1+d}{d}+c.\]
To see why, note that, if $a=\binom{m-1+d}{d}$ then $b\leq c$ and so, by \eqref{ineq}, we get $b^{\l d\r}\leq c^{\l d\r}$; which shows (2). On the other hand, for the other assumption, suppose $a+b< \binom{m-1+d}{d}+c$ and let $a'$ be the largest integer such that $a'\leq \binom{m-1+d}{d}$ and $a'+b\leq \binom{m-1+d}{d}+c$. If the latter inequality is strict, we must have $a'=\binom{m-1+d}{d}$ (otherwise, $a'+1$ would contradict the maximality of $a'$); and so we again would get $b\leq c$, from which (2) follows using \eqref{ineq}. Thus, we may assume $a'+b= \binom{m-1+d}{d}+c$. Since $a\leq a'$ implies $a^{\l d\r}\leq a'^{\l d\r}$, it now suffices to prove (2) with $a'$ in place of $a$.

We proceed by induction on $d$. For the base case, $d=1$, note that 
$$a^{\l 1\r}=\frac{a(a+1)}{2}.$$ 
Thus,
$$a^{\l 1\r}+b^{\l 1\r}\leq \frac{a(a+1)+b(b+1)+2ab}{2}=\frac{(a+b)(a+b+1)}{2}=(a+b)^{\l 1\r}$$
To prove (2) in this case, write $m=a+i=b+j$ with $i,j\geq 0$, then
$$
m c=(a+i)(a-j)=a(a+i-j)-ij=ab-ij\leq ab,
$$
and so
$$a^{\l 1\r}+b^{\l 1\r}=\frac{(a+b)^2-2ab+a+b}{2}\leq \frac{(m+c)^2-2mc+m+c}{2}=m^{\l 1\r}+c^{\l 1\r}.$$

We now assume $d>1$. By \eqref{form1} and $\eqref{form2}$, we can write
$$a=\binom{s-1+d}{d}+A \quad \text{ with } 0\leq A < \binom{s-1+d}{d-1},$$
where $s>0$ and 
$$b=\binom{t-1+d}{d}+B \quad \text{ with } 0< B \leq \binom{t-1+d}{d-1},$$
where $t\geq 0$. We will construct new integers $a_1$ and $b_1$ such that
\begin{align}\label{one}
0\leq b_1< b\leq a&<a_1\leq \binom{s+d}{d},\\
\label{two}
a+b&=a_1+b_1,\\
\label{three}
a^{\l d\r}+b^{\l d\r}&\leq a_1^{\l d\r}+b_1^{\l d\r}.
\end{align}
We consider two cases:

$$\text{\bf Case 1:}\quad A+B<\binom{s-1+d}{d-1}.$$ In this case, we set 
$$a_1=\binom{s-1+d}{d}+A+B\quad \text{ and }\quad b_1=\binom{t-1+d}{d}.$$
Clearly \eqref{one} and \eqref{two} are satisfied. By \eqref{use}, we have
$$a^{\l d\r}=\binom{s+d}{d+1}+A^{\l d-1\r}\quad \text{ and }\quad b^{\l d\r}=\binom{t+d}{d+1}+B^{\l d-1\r}$$
and 
$$a_1^{\l d\r}=\binom{s+d}{d+1}+(A+B)^{\l d-1\r}\quad \text{ and }\quad b_1^{\l d\r}=\binom{t+d}{d+1}.$$
Since by induction $A^{\l d-1\r}+B^{\l d-1\r}\leq (A+B)^{\l d-1\r}$, from the above equalities, we get
\eqref{three}.

$$\text{\bf Case 2:}\quad A+B=\binom{s-1+d}{d-1}+e$$ for some $e\geq 0$. In this case, we set
$$a_1=\binom{s+d}{d} \quad \text{ and }\quad b_1=\binom{t-1+d}{d}+e.$$
We clearly have \eqref{one}. By \eqref{bin}, we have 
$$a_1+b_1=\binom{s+d-1}{d}+\binom{s+d-1}{d-1}+b+A-\binom{s-1+d}{d-1}=a+b.$$
This yields \eqref{two}. Now, since $b\leq a$, we have $t\leq s$, and so $B< \binom{s-1+d}{d-1}$. By induction, we have
\begin{equation}\label{eq:ABd1}
A^{\l d-1\r}+B^{\l d-1\r}\leq \binom{s-1+d}{d-1}^{\l d-1\r}+e^{\l d-1\r}.
\end{equation}
On the other hand, by~\eqref{rep2} and~\eqref{use}, we have
\begin{equation}\label{eq:a1b1}
a_1^{\l d\r}=\binom{s+d+1}{d+1} \quad \text{ and }\quad  b_1^{\l d\r}=\binom{t+d}{d+1}+e^{\l d-1\r},
\end{equation}
where, for the latter equality, we have  used the fact that $0\leq e<\binom{t-1+d}{d-1}$. Formulas~\eqref{eq:ABd1},~\eqref{eq:a1b1}, ~\eqref{bin} and~\eqref{use} yield
\begin{align*}
a^{\l d\r}+b^{\l d\r} & =\binom{s+d+1}{d+1}-\binom{s+d}{d}+\binom{t+d}{d+1}+A^{\l d-1\r}+B^{\l d-1\r} \\
&\leq a_1^{\l d\r} -\binom{s+d}{d}+b_1^{\l d\r} +\binom{s-1+d}{d-1}^{\l d-1\r} \\
& =a_1^{\l d\r}+b_1^{\l d\r},
\end{align*}
and hence we obtain \eqref{three}. 

Iterating this construction, obtaining $(a_{i+1},b_{i+1})$ from $(a_i,b_i)$ satisfying \eqref{one} to  \eqref{three}, one eventually finds $\ell_1\leq \ell_2$ such that $a_{\ell_1}=\binom{m-1+d}{d}$ (which implies $b_{\ell_1}=c$) and $a_{\ell_2}=a+b$ (which implies $b_{\ell_2}=0$). This shows that
$$a^{\l d\r}+b^{\l d\r}\leq a_{\ell_1}^{\l d\r}+b_{\ell_1}^{\l d\r}=\binom{m-1+d}{d}^{\l d\r}+c^{\l d\r}$$
 and
\[a^{\l d\r}+b^{\l d\r}\leq a_{\ell_2}^{\l d\r}+b_{\ell_2}^{\l d\r}=(a+b)^{\l d\r}.\qedhere\]
\end{proof}

\begin{lemma}\label{technical}
Let $m$ and $d$ be positive integers. Suppose $a_1\leq \cdots \leq a_t$ and $b_1, \dots, b_s$ are sequences of nonnegative integers such that 
$$b_1\leq b_2= \cdots =b_s=\binom{m-1+d}{d}$$
and $b_s\geq a_i$ for all $i\leq t$. If $a_1+\cdots+a_t\leq b_1+\cdots+b_s$, then
$$a_1^{\langle d\rangle}+\cdots+a_t^{\langle d \rangle}\leq b_1^{\langle d\rangle}+\cdots+b_s^{\langle d\rangle}.$$
\end{lemma}
\begin{proof}
We proceed by induction on $(t,s)$ using the lexicographic order. 
The case $t=1$ follows from~ \eqref{ineq}.
The case $s=1$ follows from~Lemma~\ref{rep}. 
Thus, we assume that $t,s>1$. We now consider two cases:
\begin{enumerate}
\item[\underline{Case 1.}] Suppose $b_1\geq a_1$. Then the sequences $a_2\leq \cdots\leq a_{t}$ and $b_1-a_1\leq b_2\leq \cdots \leq b_s$ satisfy our hypothesis. By induction, 
 $$a_2^{\langle d\rangle}+\cdots+a_{t}^{\langle d\rangle} \leq (b_1-a_1)^{\langle d\rangle}+b_{2}^{\langle d\rangle}+\cdots + b_s^{\langle d\rangle}.$$
Using that $a_1^{\langle d\rangle}+(b_1-a_1)^{\langle d\rangle}\leq b_1^{\langle d\rangle}$, which follows from Lemma~\ref{rep}, we get the desired inequality for the original sequences.
\item[\underline{Case 2.}] Suppose $b_1< a_1$. When $s=2$, we must have that $a_2+\cdots +a_{t}\leq b_2$ and so
$$a_1^{\langle d\rangle}+\cdots+a_t^{\langle d\rangle}\leq a_1^{\langle d\rangle} +(a_2+\cdots +a_{t})^{\langle d\rangle} \leq b_1^{\langle d\rangle}+b_2^{\langle d\rangle},$$
where the first inequality follows from part (1) of Lemma~\ref{rep} and the second from part (2). So we assume that $s>2$. If it happens that $a_1+\cdots+a_{t}\leq b_2+\cdots+b_{s}$, then we are done by induction. So we can assume that
\begin{equation}\label{cas}
a_2+\cdots+a_{t} > b_3+\cdots+b_{s}.
\end{equation}
We have that
$$b_{1} + b_2\geq a_1 +\left(a_2+\cdots +a_{t}-b_3-\cdots -b_{s}\right).$$
It follows from Lemma~\ref{rep}, using \eqref{cas}, that
$$b_{1}^{\langle d\rangle} +b_2^{\langle d\rangle} \geq a_1^{\langle d\rangle} + \left(a_2+\cdots +a_{t}-b_3-\cdots -b_{s}\right)^{\langle d\rangle}.$$
Thus, it suffices to see that 
$$\left(a_2+\cdots +a_{t}-b_3-\cdots -b_{s}\right)^{\langle d\rangle} + b_3^{\langle d\rangle}+\cdots + b_{s}^{\langle d\rangle}  \geq a_2^{\langle d\rangle}+\cdots +a_{t}^{\langle d\rangle},$$
but this follows by induction. Indeed, let $a_1'=a_2, \dots, a'_{t-1}=a_t$ and $b'_1=a_2+\cdots +a_{t}-b_3-\cdots -b_{s}$ and $b'_2=b_3,\dots,b'_{s-1}=b_s$. After these relabelings, the desired inequality is precisely $$a'^{\langle d\rangle}_1+\cdots+a'^{\langle d\rangle}_{t-1}\leq b'^{\langle d\rangle}_1+b'^{\langle d\rangle}_2+\cdots b'^{\langle d\rangle}_{s-1}.$$
Clearly, $a'_1\leq \cdots\leq a'_{t-1}$, $\binom{m-1+d}{d}=b'_2=\cdots=b'_s\geq a'_i$ and $a'_1+\cdots+a'_{t-1}\leq b'_1+\cdots+b'_{s-1}$. Hence, to prove the above displayed inequality using induction, it only remains to check that $b'_1\leq \binom{m-1+d}{d}$. To see this, note that, since $b_1<a_1$, we have $a_2+\cdots +a_{t}\leq b_2+\cdots +b_s$; in other words (after moving $b_3+\cdots+b_s$ to the left-hand side) $b'_1\leq b_2=\binom{m-1+d}{d}$, as desired.\qedhere
\end{enumerate}
\end{proof}

Now, one way to make the paper formally correct is to replace (3.5) with Lemma~\ref{rep} above, and replace Lemma 3.12 with Lemma \ref{technical} above; furthermore:
\begin{enumerate}
\item In line 14 of the proof of Proposition 3.8, replace the use of (3.5) with \eqref{ineq} above.
\item In the proof of Proposition 3.13, replace the use of Lemma 3.12 with Lemma~\ref{technical} above. The assumptions of Lemma \ref{technical} are indeed satisfied in Proposition 3.13 as expressed by the inequalities in line~17 of its proof. More precisely, if we set $a_k=H_{\bar a}^{i+1,k}(d)$, for $k=1,\dots,n$, and $b_k=H_{\bar\mu}^{i+1,k-1+j}(d)$, for $k=1,\dots,n+1-j$, where $j$ is such that $L_0+\cdots+L_{j-1}\leq i+1\leq L_0+\cdots+L_j$, then the inequalities in line~17 of the proof of Proposition 3.13 say  that 
$$b_1\leq b_2=\cdots=b_{n+1-j}=\binom{m-1+d}{d}.$$ 
Also, by definition $H_{\bar a}^{i+1,k}(d)\leq\binom{m-1+d}{d}$, and so  $a_k\leq b_{n+1-j}$ for all $k=1,\dots,n$. The induction hypothesis on $d$ before display (3.19) implies
that $$H_{\bar a}^{i+1,1}(d)+\cdots+H_{\bar a}^{i+1,n}(d)\leq H_{\bar \mu}^{i+1,j}(d)+\cdots+H_{\bar \mu}^{i+1,n}(d),$$ 
which translates to $a_1+\dots+a_n\leq b_1+\cdots+b_{n+1-j}$, and so Lemma~\ref{technical} yields $$a_1^{\langle d\rangle}+\dots+a_n^{\langle d\rangle}\leq b_1^{\langle d\rangle}+\cdots+b_{n+1-j}^{\langle d\rangle}.$$ This yields the inequality in display (3.19), which is the only place in the proof at which Lemma 3.12 was used.
 \end{enumerate}

\bibliographystyle{abbrvnat}
\bibliography{bibdata}

\end{document}